\documentclass[14pt,a4paper,oneside]{amsart}
\usepackage[english]{babel}
\usepackage[utf8]{inputenc}
\usepackage{amsmath,amssymb}
\usepackage{amsthm,mathtools}
\usepackage[foot]{amsaddr}
\usepackage{graphicx}
\allowdisplaybreaks
\usepackage{bbm}

\usepackage{color}

\newcommand{\blue}[1]{{\color{blue}#1}}

\usepackage{geometry}
\usepackage{tikz}
\usepackage{tikz-cd}
\usetikzlibrary{intersections,decorations.markings,calc,angles}
\usepackage{mathrsfs}
\usepackage[colorinlistoftodos, textwidth=25mm]{todonotes}
\usepackage{enumitem}

\usepackage{hyperref}
\hypersetup{hypertex=true,colorlinks=true,linkcolor=cyan,anchorcolor=cyan,citecolor=cyan}

\newcommand\abs[1]{|#1|}

\newcommand{\R}{\mathbb{R}}

\newcommand{\inv}{^{-1}}

\newcommand{\mbz}{\mathbb{Z}}

\newcommand{\mbn}{\mathbb{N}}

\newcommand{\mcc}{\mathcal{C}}

\newcommand{\mcg}{\mathcal{G}}

\newcommand{\mcs}{\mathcal{S}}

\newcommand{\mme}{\mathrm{e}}
\newcommand{\mmi}{\mathrm{i}}

\newcommand{\fot}{\frac{1}{2}}

\DeclareMathOperator{\RE}{Re}

\newmuskip\pFqmuskip
\newcommand*\pFq[6][8]{%
	
}

\usepackage{verbatim}

\theoremstyle{plain}
\newtheorem{thm}{Theorem}[section]
\newtheorem*{thm*}{Theorem A}
\newtheorem*{thm**}{Theorem B}
\newtheorem{lem}[thm]{Lemma}

\theoremstyle{definition}

\theoremstyle{remark}

\title{Structure of large quadratic character sums}
\author{Zikang Dong}
\author{Weijia Wang}
\author{Hao Zhang}
\address[Zikang Dong]{School of Mathematical Sciences, Tongji University, Shanghai 200092, P. R. China}
\address[Weijia Wang]{Yanqi Lake Beijing Institute of Mathematical Sciences and Applications $\&$ Yau Mathematical Sciences Center, Tsinghua University, Beijing 101408, P. R. China}
\address[Hao Zhang]{School of Mathematics, Hunan University, Changsha 410082, P. R. China}
\email{zikangdong@gmail.com}
\email{weijiawang@tsinghua.edu.cn}
\email{zhanghaomath@hnu.edu.cn}

\begin{document}
	
	\maketitle
	\section*{Abstract}
	In this article, we study the distribution of large quadratic character sums. Based on the recent work of Lamzouri, we obtain the structure results  of quadratic characters with large character sums which generalize the previous work of Bober, Goldmakher, Granville and Koukoulopoulos to the family of quadratic characters. In comparison, our error terms are better than theirs.
	\bigskip
	\section{\blue{Introduction}}

 Let $q$ be an integer and $\chi\pmod q$ be any Dirichlet character modulo $q$. Define the large character sum of $\chi$ by
 $$M(\chi):=\max_{t\le q}\bigg|\sum_{n\le t}\chi(n)\bigg|.$$
 The values of $M(\chi)$ plays a fundamental role in many areas of number theory, for example in the distribution of quadratic residues,  and particularly in the classical question of finding an upper bound for the least quadratic nonresidue of a large modulus (see \cite{Bob,BoGo16,GrMa}).
 We also define by $0<N_\chi<q$ the point where $M(\chi)$ arrives at the maximal size:
 $$M(\chi)=\bigg|\sum_{n\le N_\chi}\chi(n)\bigg|.$$
 When $|d|\le x$ is a fundamental discriminant, then there is a unique real primitive character modulo $|d|$, and we can define the corresponding character sums $M(\chi_d)$  and  $N_{\chi_d}$. The goal of this article is to gain insight into the structure of those quadratic characters $\chi_d$ for which $M(\chi_d)$ is large.

In 1918, P\'{o}lya and Vinogradov independently proved that
$$M(\chi) \ll \sqrt{q} \log q.$$
In 1977, assuming the Generalized Riemann Hypothesis (GRH), Montgomery and Vaughan \cite{MV2} proved that $$M(\chi)\ll \sqrt{q}\log_2q.$$ Here and throughout we shall denote by $\log_k$ the $k$-th iteration of the natural logarithm. This upper bound is the best possible up to a constant factor in view of Paley's \cite{Pa} omega result \eqref{Paley}. In 2007, Granville and Soundararajan \cite{GrSo2} refined Montgomery and Vaughan's conditional result by showing that 
\begin{equation}\label{UpperGRH}
 M(\chi)\leq \Big(2 C_{\pm}+o(1)\Big) \sqrt{q}\log_2 q,
 \end{equation}
 where $C_{-}={\rm e}^{\gamma}/\pi$ if $\chi$ is odd, and $C_{+}={\rm e}^{\gamma}/(\sqrt{3}\pi)$ if $\chi$ is even. Granville and Soundararajan \cite{GrSo2} conjectured that the true extreme values of $M(\chi)$ only differ by a factor 2, which means
\begin{equation}\label{ConjectureGrSo}
 M(\chi)\le\big(C_{\pm}+o(1)\big) \sqrt{q}\log_2 q, 
 \end{equation}
 for all primitive characters $\chi$. In view of this, there is not much space to improve their upper bound \eqref{UpperGRH}. However, there are some improvements for some special family of characters. For example, the breakthrough of Granville and Soundararajan \cite{GrSo2}  
for any character of fixed order, which were subsequently further improved by Goldmakher \cite{GOLD} and Lamzouri and Mangerel \cite{LaMa}.

For the omega results of $M(\chi)$, in 1932, Paley \cite{Pa} showed that there exist infinitely many moduli $q$ and quadratic characters $\chi\pmod q$ such that
\begin{align}\label{Paley}
  M(\chi)\gg\sqrt q\log_2q.  
\end{align}
This was improved by Bateman and Chowla \cite{BaCh} in 1950, who proved the existence of an infinite sequence of moduli $q$, and primitive quadratic characters $\chi \pmod q$, such that 
\begin{equation}\label{LowerOmega}
M(\chi) \geq  \Big(\frac{{\rm e}^{\gamma}}{\pi}+o(1)\Big) \sqrt{q}\log_2 q.
\end{equation}
In view of Granville and Soundararajan's conjecture \eqref{ConjectureGrSo}, this is the best possible up to the error term. Their result was extended to the family of primitive characters modulo a large prime $q$ by several authors (see for example, Theorem 3 of \cite{GrSo2}). For results of fixed order characters, we refer to \cite{GL1,GL2}.

Define the normalized large character sum 
 $$m(\chi):=\frac{M(\chi)}{\mme^\gamma\sqrt q/\pi},$$
 where $\gamma$ is the Euler constant. 
In 1979, Montgomery and Vaughan \cite{MV79}  studied the distribution of $m(\chi)$ over families of Dirichlet characters, and  showed that $m(\chi)$ is bounded 
  for most characters, which implies $M(\chi)\ll \sqrt{q}$ mostly often. Let $q$ be a large prime and 
$$ \Phi_q(\tau):= \frac{1}{\varphi(q)} \# \{\chi\neq \chi_0 \ (\bmod \ q) : m(\chi)>\tau\},$$
be the distribution function, where $\varphi(q)$ is Euler's totient function. It follows from Montgomery and Vaughan's work \cite{MV79} that 
$$ \Phi_q(\tau) \ll_A \tau^{-A}, $$ for any constant $A\geq 1$. 
This estimate was improved by Bober and Goldmakher \cite{BoGo} for fixed $\tau$, and subsequently by Bober, Goldmakher, Granville and Koukoulopoulos \cite{BGGK} who showed that uniformly for $2\leq \tau \leq \log_2 q-M$ (where $M\geq 4$ is a parameter) we have 
\begin{equation}\label{BGGK}
\exp\bigg(-\frac{\mme^{\tau+A_0-\eta}}{\tau}\big(1+O(E_1(\tau, M))\big)\bigg) \leq \Phi_q(\tau)\leq \exp\bigg(-\frac{\mme^{\tau-2-\eta}}{\tau}\big(1+O(E_2(\tau))\big)\bigg), 
\end{equation}
where $E_1(\tau, M)= (\log \tau)^2/\sqrt{\tau}+\mme^{-M/2}$, $E_2(\tau)=(\log \tau)/\tau$, $\eta:= \mme^{-\gamma}\log 2$, and $A_0=0.088546...$ is an explicit constant. The lower bound is highly based on the distribution of values of Dirichlet $L$-functions proved by Granville and Soundararajan \cite{GrSo06}. The proof for the upper bound is much more complex, which is based on the fact that the main contribution in P\'olya's Fourier expansion (i.e. Lemma \ref{polyatruc}) comes from the part of friable numbers for suitably large $y$:
\begin{align}\sum_{1\le|n|\le z}\frac{\chi(n)(1-\mme(n\alpha))}{n}\approx\sum_{1\le|n|\le z\atop n\in {\mathcal S(y)}}\frac{\chi(n)(1-\mme(n\alpha))}{n},\label{smoothinput}\end{align} where ${\mathcal S}(y)$ is defined in \eqref{friableset}.

Granville and Soundararajan \cite{GrSo2} showed when $M(\chi)$ gain large values then $\chi$ must ``pretend" to be a character of small conductor and opposite parity. While Bober, Goldmakher, Granville and Koukoulopoulos \cite{BGGK} showed that for almost all characters $\chi\pmod q$ with large values of $M(\chi)$, $\chi$ must be odd and pretend to be 1. We say multiplicative functions $f(\cdot)$ pretends to be $g(\cdot)$ if 
$${\mathbb D}(f,g;y)^2\ll\log_2 y,$$
where ${\mathbb D}(f,g;y)$ is defined by \eqref{distance}. Moreover, they also established accurate estimates for $\sum_{n\le N}\chi(n)$ for any positive number $N$. See \cite[Theorem 2.1-2.3]{BGGK}. The aim of this article is to generalize these results to the family of quadratic character sums, which will be presented later in Theorems \ref{thm1.1}-\ref{thm1.4}.

For the family of quadratic characters, Montgomery and Vaughan \cite{MV79} showed that for most primes $p\leq x$ we have $\max_{t}\Big|\sum_{n\leq t}\left(\frac{n}{p}\right)\Big|\ll \sqrt{p}$. 
Denote by $\mathcal{F}(x)$ the set of fundamental discriminants up to $x$
$$\mathcal{F}(x):=\{|d|\le x:\;d\;{\rm fundamental \;discriminant}\}.$$
Separate $\mathcal{F}(x)$ into two parts $\mathcal{F}^{\pm}(x)$ by whether the attached real  primitive character $\chi_d$ is odd or even. Define the distribution function by
$$\Psi^{\pm}_x(\tau):=\frac{1}{\#\mathcal{F}^{\pm}(x)}  \#\{d\in \mathcal{F}^{\pm}(x) : m(\chi_d)>\tau\}.$$
Uniformly for $\tau$ in the range $2\leq \tau\leq (1+o(1))\log\log x$ in the case of  $\Psi^{-}_x(\tau)$, and $2\leq \tau\leq (1/\sqrt{3}+o(1))\log\log x$ in the case of  $\Psi^{+}_x(\tau)$, Lamzouri \cite{La2022} showed the following results. The lower bounds are highly based on the distribution of values of $L(1,\chi_d)$ showed by Granville and Soundararajan \cite{GrSo}. The upper bounds rely on the result for quadratic character sums similar to \eqref{smoothinput}. More precisely, he established Lemma \ref{estimateforS} below, which is the main ingredient of \cite{La2022}. He used a quite different method (the quadratic large sieve) to prove it from the case of characters $\chi\pmod q$. His method can lead to an easier proof for the upper bound of \eqref{BGGK}, but in a slightly smaller range of $\tau$.
\begin{thm*}\cite[Theorem 1.1]{La2022}\label{Main}
Let $\eta= \mme^{-\gamma}\log 2$, and $x$ be a large real number. Uniformly for $\tau$ in the range $2\leq \tau \leq \log_2 x-\log_4 x+\log_5 x- C$ (where $C>0$ is a suitably large constant) we have 
$$ \exp\bigg(-\frac{\mme^{\tau-\eta
-B_0}}{ \tau} \bigg(1+O\bigg(\bigg{(\log \tau)^2 }{\sqrt{\tau}}\bigg)\bigg)\bigg)\leq  \Psi^{-}_x(\tau)\leq \exp\bigg(-\frac{\mme^{\tau-\eta -\log 2 - 2}}{ \tau} \bigg(1+O\bigg(\frac{\log \tau }{\tau}\bigg)\bigg)\bigg),$$
where 
\begin{equation}
B_0= \int_0^1 \frac{\tanh y}{y} dy + \int_1^{\infty} \frac{\tanh y-1}{y}dy=0.8187...
\end{equation}
\end{thm*}
\begin{thm**}\cite[Theorem 1.2]{La2022}\label{MainPositive}
Let $B_0$ be the constant in Theorem A. There exist positive constants $C_1$ and $C_2$ such that uniformly for $\tau$ in the range $2\leq \tau \leq (\log_2 x-\log_4 x+\log_5 x- C_1)/\sqrt{3}$ we have 
$$ \exp\bigg(-\frac{\mme^{\sqrt{3}\tau-B_0}}{\sqrt{3}\tau} \bigg(1+O\bigg(\frac{1 }{\tau}\bigg)\bigg)\bigg)\leq \Psi^+_x(\tau)\ll \exp\bigg(-\frac{\mme^{\sqrt{3}\tau}}{ \tau^{C_2}}\bigg).$$
\end{thm**}
Based on Lamzouri's results, we are able to show the structure of characters with large values of $M(\chi_d)$ as stated in \cite[Remark 1.5]{La2022}, which generalizes the results of Bober, Granville, Goldmakher and Koukoulopoulos \cite[Theorems 2.1-2.3]{BGGK}. However, our error terms are better due to the fact that quadratic characters only take  real values.
	\begin{thm}\label{thm1.1} Let $x$ be a large real number and $1\le\tau\le\log_2 x-\log_4x+\log_5x- C$ for a sufficiently large absolute constant $C$. There exists a set ${\mathcal C}_x(\tau)\subset  \{d\in{\mathcal F}(x): m(\chi_d)>  \tau\}$ of cardinality 
	$$
	\#{\mathcal C}_x(\tau) = \left(1+ O\left( \mme^{-\mme^\tau/\tau}\right)\right) 
	\cdot  \#  \left\{d\in\mathcal{F}(x) : m(\chi_d)>  \tau\right\}
	$$
	such that the following holds: If $d\in {\mathcal C}_x(\tau)$, then
	\begin{enumerate}
		\item $\chi_d$ is an odd character.
		\item Let $\alpha=N_{\chi_d}/|d|\in[0,1]$, and let $a/b$ be the reduced fraction for which $|\alpha-a/b|\le 1/(b\tau^{10})$ with $b\le\tau^{10}$. Let $b_0 =  b$ if $b$ is prime, and $b_0=1$ otherwise. Then 
		\begin{equation}\label{eq:logtau34}
		    \sum_{p\le e^\tau \atop p\neq b_0} \frac{1-\chi_d(p)}{p} \ll \sqrt{\frac{\log\tau}{\tau}}
		\end{equation}
		\begin{equation}\label{eq:mchidsqrttau}
		 m(\chi_d)  =\mme^{-\gamma}	\frac{b_0}{\phi(b_0)}\bigg|\prod_{p\neq b_0} \bigg(1-\frac{\chi_d(p)}{p}\bigg)^{-1}\bigg| 
			+  O\left(\sqrt{\tau\log\tau}\right) .
		\end{equation}
		\end{enumerate}
	\end{thm}

Let $\mme(x):=\mme^{2\pi{\rm i}x}$. For any character $\chi\pmod q$, denote by ${\mathcal G}(\chi)$ the Gauss sum
$${\mathcal G}(\chi):=\sum_{n\le q}\chi(n)\mme(n/q).$$
 
	\begin{thm}\label{thm1.2} Let $x, \tau, \alpha, a, b, b_0$ and $\chi_d\in {\mathcal C}_x(\tau)$ be as in Theorem \ref{thm1.1}. Given $\beta\in[0,1]$,  let $k/\ell$ be the reduced fraction for which $|\beta-k/\ell|\le 1/(\ell \tau^{10})$ with $\ell\le\tau^{10}$. Define $u>0$ by $|\beta-k/\ell|=1/(\ell \mme^{\tau u})$.
		\begin{enumerate}
			\item If $b_0=1$, then
			\[
			\frac{\mme^{-\gamma}\pi \mmi}{{\mathcal G}(\chi_d)} \sum_{n\le \beta |d|} \chi_d(n)
			= \begin{cases}
				 \tau \big(1-P(u)\big)
				+ O( (\tau\log\tau)^{\fot}) 
				&\mbox{if $\ell=1$},\\
				 \tau
				+  O( (\tau\log\tau)^{\fot} )
				&\mbox{if $\ell>1$} .
			\end{cases}
			\]
			\item If $b_0=b$, then
			\[
			\frac{\mme^{-\gamma}\pi \mmi}{{\mathcal G}(\chi_d)} \sum_{n\le \beta |d|} \chi_d(n)
			= \lambda\tau\frac{1-b\inv}{1- {\chi}_d(b)b\inv}+ O( (\tau\log\tau)^{\fot}) ,
			\]
			where
			\[
			\lambda =   \begin{cases}
				1-P(u)  &\mbox{if $\ell=1$}, \\ 
				 1+ P(u) \left(\frac{ {\chi}_d(b)}b\right)^{v-1} \frac{1- {\chi}_d(b)}{b-1} 
				&\mbox{if $\ell=b^v$, $v\ge1$}, \\ 
				1
				&\mbox{otherwise}.
			\end{cases}
			\]
			Moreover, if we take $\alpha=\beta$ and write $|\alpha-a/b|=1/(b\mme^{\tau u_0})$, then we have that 
			\begin{equation}\label{eq:pu0tau}
		(1-P(u_0)) \cdot \frac{|1-\chi_d(b)|^2}{b^2} \ll (\log\tau)^{\fot}\tau^{-\fot}.
		\end{equation}
		\end{enumerate}
	\end{thm}

	\begin{thm}\label{thm1.3}  Let $d$ be a fundamental discriminant and $1\le\tau\le \frac{1}{\sqrt{3}} (\log_2x-\log_4x+\log_5x- C)$ for a sufficiently large constant $C$. There exists a set ${\mathcal C}_x^+(\tau)\subset  \{d\in {\mathcal F}^+(x):\ m(\chi_d)>  \tau\}$ of cardinality 
	$$
			\#{\mathcal C}_x^+(\tau) = \left(1+ O\left( \mme^{-\mme^\tau/\tau}\right)\right) 
			\cdot  \#  \left\{d\in {\mathcal F}^+(x) : \chi_d(-1)=1,\  m(\chi_d)>  \tau\right\}
		$$
		for which the following statements hold. If $d\in {\mathcal C}_x^+(\tau)$ then
		\begin{enumerate}
			\item If $\alpha=N_{\chi_d}/|d|\in[0,1]$, and $a/b$ is the reduced fraction for which $|\alpha-a/b|\le 1/(b\tau^{10})$ with $b\le\tau^{10}$, then $b=3$.
			\item We have that
			$$
				\sum_{\substack{p\le \mme^\tau \\ p\neq 3}} \frac{({\frac{p}{3})-\chi_d(p)}}{p} \ll { \frac{\log\tau}{\tau} },$$
    and
			$$
				m(\chi) 
				= \frac{\mme^{-\gamma}\sqrt{3}}{2} \bigg|{L\left(1,\chi_d\Big(\frac{\cdot}{3}\Big)\right) }\bigg| + O(\log \tau) .
			$$
		\end{enumerate}
   \end{thm}

	\begin{thm}\label{thm1.4}
			 Given $\beta\in[0,1]$,  let $k/\ell$ be the reduced fraction for which $|\beta-k/\ell|\le 1/(\ell \tau^{10})$ with $\ell\le\tau^{10}$. Define $u>0$ by $|\beta-k/\ell|=1/(\ell \mme^{\sqrt{3}\tau u})$. Then
			\[
			\frac{\mme^{-\gamma}\pi}{{\mathcal G}(\chi_d)} \sum_{n\le \beta |d|} \chi_d(n)
			= \begin{cases}  \tau P(u)\cdot\Big( \frac{k}{3} \Big)\cdot 
				\frac{{\chi}_d(3^{v-1}) }{3^{v-1}}  
				+ O( \log \tau)    &\mbox{if $\ell=3^v$ for some $v\ge1$}, \\ 
				\\
				O( \log \tau)    &\mbox{otherwise}.
			\end{cases}
			\]
	\end{thm}
	This article is organized as follows. We will present some preliminary lemmas in \S\ref{sec2}. We will prove Theorems \ref{thm1.1}-\ref{thm1.4} separately in   \S \ref{sec3}-\ref{sec6}.

	\section{\blue{Preliminary Lemmas}}\label{sec2}
In this section, we give some lemmas that we will use later. The first one is called the P\'olya truncation formula. 

\begin{lem}\label{polyatruc}Let $\chi\pmod q$ be any Dirichlet character and $\alpha<1$. Then we have
$$\sum_{n\le\alpha q}\chi(n)=\frac{\mathcal G(\chi)}{2\pi \mmi}\sum_{1\le|n|\le z}\frac{\overline\chi(n)(1-\mme(-n\alpha))}{n}+O\bigg(1+\frac{q\log q}{z}\bigg).$$
\end{lem}
\begin{proof}
This is \cite[p.311, Eq. (9.19)]{MVbook}.
\end{proof}

Let $P_+(n)$ be the largest prime factor of $n$.
Denote by ${\mathcal S}(y)$ the set of $y$-friable numbers
\begin{align}{\mathcal S}(y):=\{n\in{\mathbb Z}^*:P_+(n)\le y\}\label{friableset}
\end{align}
and 
$$S_{y,z}(\chi):=\max_{\alpha\in[0,1]}\bigg|\sum_{n\le z\atop n\notin {\mathcal S}(y)}\frac{\chi(n)\mme(n\alpha)}{n}\bigg|.$$
Lamzouri \cite{La2022} gives the distribution of $S_{y,z}(\chi_d)$ for quadratic characters $\chi_d$.

\begin{lem}\label{estimateforS}
 Let $x$ be large. There exists a constant $c>0$ such that for all real numbers $2\leq y\leq c\log x\log_4 x/(\log_3x)$  and $1/\log y\leq \delta\leq 4$, we have
$$\frac{1}{\#{\mathcal F}(x)}\#\{d\in{\mathcal F}(x):S_{y,x^{21/40}}(\chi_d)>\mme^\gamma\delta\} \ll \exp\left(-\frac{\delta^2y}{2 \log y} \left(1+O\left(\frac{\log_2 y}{\log y}+ \frac{\log _4x}{\delta\log_2x\log_3x}\right)\right)\right).$$

\end{lem}
\begin{proof}
This is \cite[Theorem 2.1]{La2022}.
\end{proof}

The following two estimates are classical results of sums over friable numbers.

\begin{lem}\label{friablesum}
   For $y\ge2$ and $u\ge1$, we have
   $$\sum_{n>y^u\atop n\in {\mathcal S}(y)}\frac 1n\ll\frac{\log y}{u^u}+\mme^{-\sqrt{\log y}}.$$
\end{lem}

\begin{proof}
This is \cite[Lemma $3.2$]{BGGK}.
\end{proof}

\begin{lem}\label{friablenumber}
   For $y\ge2$ and $u>0$, we have
   $$\sum_{n\le y^u\atop n\in {\mathcal S}(y)}\frac 1n=\mme^\gamma 
 P(u)\log y+O(1),$$
 where $$P(u):=\mme^{-\gamma}\int_{0}^{u}\rho(t)dt$$ 
 with $\rho(\cdot)$ the Dickman-de Bruijn function.
 In particular, $\lim_{u\to\infty}P(u)=1$
\end{lem}
\begin{proof}
This is \cite[Lemma 3.3]{BGGK}.
\end{proof}

In order to prove Theorem \ref{thm1.1}, we need to estimate the difference between the sum of $\frac{\chi(n)\mme(n\alpha)}{n}$ and the sum of $\frac{\chi(n)\mme(na/b)}{n}$ with $a/b$ close to $\alpha$. When $b$ is small, we have the following result.
 \begin{lem}\label{alphaapprox}
Let $y\ge2$, $z\ge(\log y)^5$, $\chi$ be a Dirichlet character, $\alpha\in\R$ and $B\in[(\log y)^5,z]$. Let $a/b$ be a reduced fraction with $b\le B$ and $|\alpha-a/b|\le 1/(bB)$. Then
\[
\sum_{1\le |n| \le z \atop n\in {\mathcal S}(y)} \frac{\chi(n)\mme(n \alpha)}{n}
	= \sum_{1\le |n| \le N \atop n\in {\mathcal S}(y)} \frac{\chi(n)\mme(n a/b)}{n}
		+ O(\log B) ,
\]
where $N=\min\{z,|b\alpha - a|^{-1}\}$.
\end{lem}
\begin{proof}
See \cite[Lemma 6.2] {GrSo2} or \cite[Lemma 4.1]{GOLD}.
\end{proof}

When $b$ is large, we have a similar estimation.

\begin{lem}\label{MVupper} 
Let $|\alpha-a/b|\le 1/b^2$, where $(a,b)=1$. For all $z, y\ge 3$, we have that
\[
\sum_{\substack{n \le z \\n\in {\mathcal S}(y)}} \frac{\chi(n)\mme(n \alpha)}{n}
	\ll \log b+ \log_2 y +\frac{(\log b)^{5/2}}{\sqrt{b}} \log y .
\]
\end{lem}
\begin{proof}
This is \cite[Corollary 2.2]{GOLD}.
\end{proof}

In \cite{GrSo2}, Granville and Soundararajan defined the distance between multiplicative functions $f$ and $g$:
\begin{align}
{\mathbb D}(f,g;y)^2 :=\sum_{p\le y}\frac{1-\RE f(p)\overline g(p)}{p}.\label{distance}
\end{align}
Denote by $\xi=\xi(\chi,y)$ the primitive character of conductor $D=D(\chi,y)\le \log y$ such that
\begin{equation}\label{eq:mathdxiy}
    {\mathbb D}(\xi;y)=\min_{d\le\log y\atop \psi\pmod d\; {\rm primitive}}{\mathbb D}(\chi,\psi;y).
\end{equation}
With the help of this notation, the sum of multiplicative function $f$ can be bounded by the distance between $f$ and $1$. 
\begin{lem}\label{functupper}
Let $f:\mathbb{N}\to \mathbb{U}$ be a multiplicative function. For $z,y\ge 1$ we have that
\[
\sum_{ \substack{ n\le z \\ n\in {\mathcal S}(y)}} \frac{f(n)}{n} \ll (\log y) \exp\{-\mathbb{D}(f,1;y)^2 / 2 \}.
\]
\end{lem}
\begin{proof}
This is \cite[Lemma 4.3]{GrSo2}.
\end{proof}

\begin{lem}\label{coprimefct}
Let $f:{\mathbb N} \to {\mathbb U}$ be a completely multiplicative function. For all $a\in{\mathbb N}$, we have 

\[
\sum_{ n\le z \atop (n,a)=1 }\frac{f(n)}{n} 
	 = \prod_{p|a} \left(1-\frac{f(p)}{p}\right)
		\sum_{n\le z }\frac{f(n)}{n}  
		+ O\bigg(\frac{a}{\varphi(a)}\sum_{p|a} \frac{\log p}{p} \bigg)  .
\]
\end{lem}
\begin{proof}
This is \cite[Lemma 7.7]{BGGK}.
\end{proof}

Similar to Lemma \ref{MVupper}, we need to estimate the sum over $\frac{\chi(n)\mme(na/b)}{n}$.

\begin{lem}\label{rationalupper}
Let $\chi$, $y$, $\xi$ and $D$ be as above, and consider a real number $z\ge1$ and a reduced fraction $a/b$ with $1\le b\le (\log y)^{1/100}$ and $(b,q)=1$. If either $D\nmid b$ or $\chi\bar{\xi}$ is even, then
\[
\sum_{1\le |n|\le z \atop n\in \mcs(y)} \frac{\chi(n) \mme(na/b) }{n}  \ll  (\log y)^{0.86} ,
\]
whereas if $D|b$ and $\chi\bar{\xi}$ is odd, then
\[
\bigg|{\sum_{ \substack{ 1\le |n|\le z \\ n\in \mcs(y) }} \frac{\chi(n)\mme(n a/b )}{n}}\bigg|
		\le (\log y) \min\left\{   \frac{2\mme^\gamma}{\sqrt{D}} \cdot   (2/3)^{\omega(b/D)}, 
			e^{-\mathbb{D}(\chi,\xi;y)^2/2+O(1)}  \right\}
	+ O((\log y)^{0.86})  .
\]
\end{lem}
\begin{proof}
This is \cite[Lemma 7.8]{BGGK}.
\end{proof}



\begin{lem}\label{remove exp cor} Let $q$ be an integer that either equals $1$ or is a prime. Let $z,y\ge 2$, and $a/b$ be a reduced fraction with $1< b\le (\log y)^{100}$. Then
\[
\sum_{\substack{ n\le z \\ n\in \mcs(y)  \\ (n,q)=1}} \frac{ \mme(na/b)}{n}
	= -\frac{{\bf 1}_{b=q}}{\phi(q)}\sum_{\substack{ n\le z \\ n\in \mcs(y)  \\ (n,q)=1}} \frac{1}{n} 
		+O\left( \left(1+\frac{{\bf 1}_{q>1}}{q} \frac{b}{\phi(b)}\right)\log_2 y \right) .
\]
\end{lem}
\begin{proof}
This is \cite[Corollary 7.10]{BGGK}.
\end{proof}

\section{\blue{Proof of Theorem \ref{thm1.1}}}\label{sec3}

Let $y=\mme^{\tau+C}$ for some sufficiently large $C$ and
\[{\mathcal C}_x(\tau):=\{d\in{\mathcal F}(x):\,\,S_{y,x^{{21}/{40}}}(\chi_d)\leq 1,m(\chi_d)>\tau\}.\]
We will prove that ${\mathcal C}_x(\tau)$ satisfies the conditions in Theorem \ref{thm1.1}. 
Define the short Euler product
$$L(s,\chi;y):=\prod_{p\le y}\bigg(1-\frac{\chi(p)}{p^s}\bigg)^{-1}=\sum_{n\in{\mathcal S}(y)}\frac{\chi(n)}{n^s},$$
and the related $k$-coprime short Euler product
$$L_k(s,\chi;y):=\prod_{p\le y\atop p\nmid k}\bigg(1-\frac{\chi(p)}{p^s}\bigg)^{-1}=\sum_{n\in{\mathcal S}(y)\atop (n,k)=1}\frac{\chi(n)}{n^s}.$$
Similarly, we can define the $k$-coprime $L$-functions
$$L_k(s,\chi):=\prod_{ p\nmid k}\bigg(1-\frac{\chi(p)}{p^s}\bigg)^{-1}=\sum_{n\ge1\atop (n,k)=1}\frac{\chi(n)}{n^s},$$though they don't converge absolutely.
So when $k=1$, we have $L_1(s,\chi;y)=L(s,\chi;y)$ and $L_1(s,\chi)=L(s,\chi)$.
The following lemma will be a key step in our proof.
\begin{lem}\label{friablelfct}
Let $E$ be defined by
$$
\Delta := \sum_{p\le y} \frac{|1-\chi(p)|}{p-1},\;\;\quad E := \bigg(1+\frac{b}{\varphi(b)}(\mme^\Delta-1)\bigg)\log_2 y .$$ We have
\begin{align*}
        m(\chi)
	= \mme^{-\gamma}\cdot \begin{cases}
		\abs{L(1,\chi;y)} + O(E)  						&\text{if $b$ is not a prime power},\\
		\abs{L(1,\chi;y)} + O\big(\sqrt{\tau E} \big)  					&\text{if $b=p^e,\ e\geq 2$}, \\
		 \frac{b}{\varphi(b)} \abs{L_b(1,\chi;y)} + O\big( \sqrt{\tau E} \big)  		&\text{if $b$ is prime}.
	\end{cases}
 \end{align*}
\end{lem}
\begin{proof}
This is \cite[Eq. (9.6)]{BGGK}.
\end{proof}
By the assumption of $\alpha$, we have $M(\chi_d)=|\sum_{n\leq \alpha |d|}\chi_d(n)|$. By P\'olya's truncating formula (Lemma \ref{polyatruc}), we have 
\[m(\chi_d)=\frac{1}{2\mme^{\gamma}}\bigg|\sum_{1\leq |n|\leq x^{{21}/{40}}}\frac{\chi_d(n)(1-\mme(-n\alpha))}{n}\bigg|+O(x^{-\frac{11}{21}}d^{\fot}\log d).\]
By our assumption, for any $d\in \mcc_x(\tau)$, we have $S_{y,x^{{21}/{40}}}(\chi_d)\leq 1$, so 
\[m(\chi_d)=\frac{1}{2\mme^{\gamma}}\bigg|\sum_{1\leq |n|\leq x^{{21}/{40}}\atop n\in \mcs(y)}\frac{\chi_d(n)(1-\mme(-n\alpha))}{n}\bigg|+O(1).\]
By Lemma \ref{friablesum} and Lemma \ref{alphaapprox}, we get 
\[m(\chi_d)=\frac{1}{2\mme^{\gamma}}\bigg|\sum_{0\neq n\in \mcs(y)}\frac{\chi_d(n)}{n}-\sum_{1\leq |n|\leq N\atop n\in \mcs(y)}\frac{\chi_d(n)\mme(an/b)}{n}\bigg|+O(\log\tau),\]
where $N=|b\alpha-a|\inv \geq \tau^{10}$. We note that $m(\chi_d)>\tau$ for all $d\in \mcc_x(\tau)$, so we have
\begin{equation}\label{eq:chineanb}
    \bigg|\sum_{0\neq n\in \mcs(y)}\frac{\chi_d(n)}{n}-\sum_{1\leq |n|\leq N\atop n\in \mcs(y)}\frac{\chi_d(n)\mme(an/b)}{n}\bigg|\geq 2\mme^{\gamma}\tau-O(\log \tau).
\end{equation}

If $b\geq \tau^{\frac{1}{100}}$, by taking $B=\tau^{10}$ in Lemma \ref{alphaapprox} and together with Lemma \ref{MVupper}, we get 
\[\sum_{1\leq |n|\leq N\atop n\in \mcs(y)}\frac{\chi_d(n)\mme(an/b)}{n}\ll \tau^{\frac{199}{200}}(\log\tau)^{\frac{5}{2}}.\]
Combining with \eqref{eq:chineanb}, we see that 
\[(1-\chi_d(-1))\bigg|\sum_{0<n\in \mcs(y)}\frac{\chi_d(n)}{n}\bigg|\geq 2\mme^{\gamma}\tau-O(\tau^{\frac{199}{200}}(\log\tau)^{\frac{5}{2}}).\]
It is easy to see that $\chi_d(-1)=-1$.

If $b\leq \tau^{\frac{1}{100}}$, by taking $a/b=0/1$ and $z\to \infty$ in Lemma \ref{rationalupper}, we get 
\[\sum_{0\neq n\in \mcs(y)}\frac{\chi_d(n)}{n}\ll\tau^{0.86}.\]
On the other hand, suppose $\xi$ is a character modulo $D$ with $D\leq \log y$ which satisfies Eq. \eqref{eq:mathdxiy}. If $b$ is not divided by $D$ or $\chi \bar{\xi}$ is even, then Lemma \ref{eq:chineanb} gives 
\[
\sum_{1\le |n|\le z \atop n\in \mcs(y)} \frac{\chi_d(n) \mme(na/b) }{n}  \ll  (\log y)^{0.86}.
\]
So we get
\[\bigg|\sum_{0\neq n\in \mcs(y)}\frac{\chi_d(n)}{n}-\sum_{1\leq |n|\leq N\atop n\in \mcs(y)}\frac{\chi_d(n)\mme(an/b)}{n}\bigg|\ll \tau^{0.86},\]
but this contradicts to Eq. \eqref{eq:chineanb}. So we must have $D|b$ and $\chi\bar{\xi}$ to be odd, and 
\[\sum_{1\le |n|\le z \atop n\in \mcs(y)} \frac{\chi_d(n) \mme(na/b) }{n}  \ll 2\mme^{\gamma}\tau D^{-\fot}(\frac{2}{3})^{\omega(b/D)}.\]
So we get 
\[\bigg|\sum_{0\neq n\in \mcs(y)}\frac{\chi_d(n)}{n}-\sum_{1\leq |n|\leq N\atop n\in \mcs(y)}\frac{\chi_d(n)\mme(an/b)}{n}\bigg|\ll 2\mme^{\gamma}\tau D^{-\fot}(\frac{2}{3})^{\omega(b/D)}.\]
Comparing to Eq. \eqref{eq:chineanb}, we must have $b=D=1$. So $\xi$ is trivial character and $\chi_d$ is odd.

Now we prove the second statement of Theorem \ref{thm1.1}. Let $d\in \mcc_x(\tau)$, for any $k\leq \tau^{10}$, by using the identity $\sum_{\ell |n}\mu(\ell)=\delta_{1,n}$, we have 
\[\sum_{n\notin S(y),n\leq x^{21/40}\atop (n,k)=1}\frac{\chi_d(n)}{n}=\sum_{\ell |k}\frac{\mu(\ell)\chi_d(\ell)}{\ell}\sum_{m\ell \notin S(y),m\ell\leq x^{{21}/{40}}}\frac{\chi_d(m)}{m}.\]
We note that for any $\ell$ we have 
\[\bigg|\sum_{x^{{21}/{40}}/\ell<m<x^{21/40}}\frac{\chi_d(m)}{m}\bigg|\leq \sum_{x^{{21}/{40}}/\ell<m<x^{{21}/{40}}}\frac{1}{m}=O(\log \ell).\]
Then we get 
\[\sum_{n\notin S(y),n\leq x^{{21}/{40}}\atop (n,k)=1}\frac{\chi_d(n)}{n}=\sum_{\ell |k}\frac{\mu(\ell)\chi_d(\ell)}{\ell}\sum_{m \notin S(y),m\leq x^{{21}/{40}}}\frac{\chi_d(m)}{m}+O\bigg(\sum_{\ell |k}\frac{|\mu(\ell)|\log \ell }{\ell}\bigg).\]
But $S_{y,x^{{21}/{40}}}(\chi_d)\leq 1$ since $d\in \mcc_x(\tau)$, so we have 
\[\sum_{n\notin S(y),n\leq x^{{21}/{40}}\atop (n,k)=1}\frac{\chi_d(n)}{n}\ll \sum_{\ell |k}\frac{|\mu(\ell)|\log \ell}{\ell}.\]
Now let $f(k):=\sum_{\ell |k}{(|\mu(\ell)|\log \ell)}/{\ell}$, and $p$ be a prime not divide $k$, then 
\begin{align*}
    f(kp^s)=&\sum_{\ell |kp^s}\frac{|\mu(\ell)|\log \ell}{\ell}=\sum_{\ell |k}\left(\frac{|\mu(\ell)|\log \ell}{\ell}+\frac{|\mu(\ell)|(\log \ell +\log p)}{p\ell}\right)\\
    =&(1+\frac{1}{p})f(n)+\frac{\log p}{p}\sum_{\ell |k}\frac{|\mu(\ell)|}{\ell}.
\end{align*}
So an induction gives that 
\[f(k)\ll \frac{k}{\varphi(k)}\sum_{p|k}\frac{\log p}{p}\ll (\log\log k)^2.\]
Hence 
\[L_k(1,\chi_d):=\sum_{(n,k)=1}\frac{\chi_d(n)}{n}=\sum_{n\leq x^{21/40}\atop n\in {\mathcal S}(y),(n,k)=1}\frac{\chi_d(n)}{n}+O((\log\log \tau)^2)=L_k(1,\chi_d;y)+O((\log\log \tau)^2).\]
So following directly from Lemma 
 \ref{friablelfct}, we can get a possibly weak version (since we don't have $E\ll\log\tau$ now) of Eq. \eqref{eq:mchidsqrttau}:
 \begin{align}\label{weakeq8}
m(\chi_d)  =\mme^{-\gamma}	\frac{b_0}{\phi(b_0)}|L_{b_0}(1,\chi_d)| 
			+  O(\sqrt{\tau E}). \end{align} With this weaker result, we are able to deduce a weaker version of Eq. \eqref{eq:logtau34}. Since Mertens' formula gives 
\[\prod_{p\leq y}\left(1-\frac{1}{p}\right)\inv=\mme^{\gamma}\tau+O(1),\]
 together with Eq. \eqref{weakeq8} we get 
\[\bigg|\prod_{p\leq y\atop p\neq b_0}\left(1-\frac{\chi_d(p)}{p}\right)\inv \left(1-\frac{1}{p}\right)\bigg|=1+O(\sqrt{E/\tau})=1+o(1).\]
By taking the logarithm, we get 
\begin{align}\label{weakeq7}
\sum_{p\leq y\atop p\neq b_0}\frac{1-\chi_d(p)}{p}\ll\sqrt{E/\tau}=o(1).\end{align}
By the definition of $E$, we have $E\ll\log\tau$. Inserting this into  Eq. \eqref{weakeq8} and \eqref{weakeq7}, we see that they are exactly  Eq. \eqref{eq:mchidsqrttau} and \eqref{eq:logtau34}. So we complete the proof of Theorem \ref{thm1.1}.

\section{\blue{Proof of Theorem \ref{thm1.2}}}\label{sec4}

Let $y=\mme^{\tau+c}$ for some $c$ large enough, then $|\beta-k/\ell|={1}/({\ell \mme^{\tau u}})={1}/({\ell y^w})$ where $w={\tau u}/{\tau+c}=u(1+O(1/\tau))$. For any $n$, if $\chi_d(n)=-1$, then there exists at least one prime number $p$ with $p^j\| n$ such that $\chi_d(p^j)=-1$. So we have $|1-\chi_d(n)|\leq \sum_{p^j\| n}|1-\chi_d(p^j)|$. Then we have 
\[\sum_{n\in S(y)\atop (n,b_0)=1}\frac{|1-\chi_d(n)|}{n}\leq \sum_{n\in S(y)}\frac{1}{n}\sum_{p^j\| n}|1-\chi_d(p^j)|\leq \sum_{p\leq y,j\geq 1}\frac{|1-\chi_d(p^j)|}{p^j}\sum_{m\in S(y)}\frac{1}{m}.\]
By Eq. \eqref{eq:logtau34}, we see that 
\[\sum_{p\leq y,j\geq 1}\frac{|1-\chi_d(p^j)|}{p^j}\leq \sum_{j\geq 0}\frac{1}{p^j}\sum_{p\leq y}\frac{|1-\chi_d(p)|}{p}\ll \tau^{-\fot}(\log\tau)^{\fot}.\]
So together with Lemma \ref{friablenumber}, we get 
\[\sum_{n\in S(y)\atop (n,b_0)=1}\frac{|1-\chi_d(n)|}{n}\ll (\tau\log\tau)^{\fot}.\]
We first consider the case $b_0=1$, then by combining 
P\'olya formula and Lemma \ref{alphaapprox}, we get
\begin{equation}\label{eq:gchid}
\begin{split}
   &\frac{\pi \mmi}{\mcg(\chi_d)}\sum_{n\leq \beta |d|}\chi_d(n)\\
    =&\frac{1}{2}\sum_{0\neq n\in \mbz\atop n\in S(y)}\frac{\chi_d(n)}{n}-\frac{1}{2}\sum_{0< |n|\leq y^w\atop n\in S(y)}\frac{\chi_d(n)\mme(-kn/\ell)}{n}+O(\log \tau)\\
=&\sum_{0<n\in S(y)}\frac{1}{n}+\frac{\chi_d(n)-1}{n}-\sum_{0<n\leq y^w\atop n\in S(y)}\frac{\mme(kn/\ell)+\mme(-kn/\ell)}{2n}-\frac{\chi_d(n)-1}{n}+O(\log\tau)\\
=&\sum_{0<n\in S(y)}\frac{1}{n}-\sum_{0<n\leq y^w\atop n\in S(y)}\frac{\mme(kn/\ell)+\mme(-kn/\ell)}{2n}+O((\tau\log\tau)^{\frac{1}{2}}).
\end{split}
\end{equation}
When $\ell>1$, then the first statement of Theorem \ref{thm1.2} follows from Lemma \ref{remove exp cor} and when $\ell=1$, it follows from Lemma \ref{friablenumber}.

Now we consider the case that $b_0=b$ is a prime number. For any $n$, we can write $n=b^sm$ with $(m,b)=1$, then by the same reason for Eq. \eqref{eq:gchid}, we have 
\begin{equation}\label{eq:2msymb}
\begin{split}
&\frac{\pi \mmi}{\mcg(\chi_d)}\sum_{n\leq \beta |d|}\chi_d(n)\\
=&\sum_{0<m\in S(y)\atop (m,b)=1, s\geq 0}\frac{\chi_d(b^s)\chi_d(m)}{b^sm}- \sum_{m\in S(y),0<|m|\leq y^w\atop (m,b)=1,s\geq 0}\frac{\chi_d(b^s)\chi_d(m)\mme(-kb^sm/\ell)}{2b^sm}+O(\log \tau)\\
=&\sum_{s\geq 0}\frac{\chi_d(b^s)}{b^s}\bigg(\sum_{0<m\in S(y)\atop (m,b)=1}\frac{1}{m}-\sum_{m\in S(y),(m,b)=1\atop 0<m\leq y^w}\frac{\mme(kmb^s/\ell)+\mme(-kmb^s/\ell)}{2m}\bigg)+O((\tau\log\tau)^{\frac{1}{2}}).
\end{split}
\end{equation}

When $\ell =1$, the equation above reduces to 
\[\frac{1}{1-\chi_d(b)b\inv}\sum_{m\in S(y),(m,b)=1\atop m>y^w}\frac{1}{m}+O((\tau\log\tau)^{\frac{1}{2}}).\]
So by Lemma \ref{friablenumber} and Lemma \ref{coprimefct} we get 
\begin{align*}
    \frac{\pi \mmi}{\mcg(\chi_d)}\sum_{n\leq \beta |d|}\chi_d(n)&=\frac{1-b\inv}{1-\chi_d(b)b\inv}\sum_{m\in S(y)\atop m>y^w}\frac{1}{m}+O((\tau\log\tau)^{\frac{1}{2}})\\
    &=\frac{1-b\inv}{1-\chi_d(b)b\inv}(1-P(w))\mme^{\gamma}\tau+O((\tau\log\tau)^{\frac{1}{2}}).
\end{align*}

When $\ell>1$ but not be a power of $b$, since $(k,\ell)=1$, then $b^sk/\ell\notin\mbz$ for any $s\in \mbn$. Moreover, if we write $b^sk/\ell$ in the reduced form, the denominator does not equal to $b$, so by Lemma \ref{remove exp cor}, the last summation in Eq. \eqref{eq:2msymb} vanishes, so with the same reason as the case $\ell=1$, we get 
\[\frac{\pi \mmi}{\mcg(\chi_d)}\sum_{n\leq \beta |d|}\chi_d(n)=\frac{1-b\inv}{1-\chi_d(b)b\inv}\mme^{\gamma}\tau+O((\tau\log\tau)^{\frac{1}{2}}).\]

When $\ell=b^t$ for a certain positive integer $t$, the summation of $\frac{1}{m}$ always equals to
\begin{equation}\label{eq:bt1m}
    \sum_{s\geq 0}\frac{\chi_d(b^s)}{b^s}\sum_{0<m\in S(y)\atop (m,b)=1}\frac{1}{m}=\frac{1-b\inv}{1-\chi_d(b)b\inv}\mme^{\gamma}\tau++O((\tau\log\tau)^{\frac{1}{2}}).
\end{equation}
For the last summation in Eq. \eqref{eq:2msymb}, we need to decompose the summation over $s$ into three parts:$0\leq s\leq t-2,s=t-1,s\geq t$. When $0\leq s\leq t-2$, the reduced form of $-\frac{kb^s}{b^t}$ is $-\frac{b^s}{b^{t-s}}$, and the denominator never to be a prime number since $t-s\geq 2$. So Lemma \ref{remove exp cor} shows that this part vanishes. When $s=t-1$, by Lemma \ref{remove exp cor}, we have 
\begin{equation}\label{eq:st-1}
    \sum_{m\in S(y),(m,b)=1\atop 0<m\leq y^w}\frac{\mme(kmb^s/\ell)+\mme(-kmb^s/\ell)}{2m}=-\frac{P(w)}{b}\mme^{\gamma}\tau+O(1).
\end{equation}
For the part $s\geq t$, by Lemma \ref{friablenumber} and Lemma \ref{coprimefct}, we get 
\begin{equation}\label{eq:sgeqt}
    \sum_{s\geq t}\frac{\chi_d(b^s)}{b^s}\sum_{m\in S(y),(m,b)=1\atop 0<|m|\leq y^w}\frac{\mme(kb^s/\ell)+\mme(-kb^s/\ell)}{2m}=\frac{(1-b\inv)\chi_d(b^t)}{(1-\chi_d(b)b\inv)b^t}P(w)\mme^{\gamma}\tau+O(1).
\end{equation}
By combining Eq. \eqref{eq:bt1m}, \eqref{eq:st-1}, \eqref{eq:sgeqt}, we get 
\begin{equation}\label{eq:mchidtau}
\frac{\pi \mmi}{\mcg(\chi_d)}\sum_{n\leq \beta |d|}\chi_d(n)=\frac{b\inv}{1-\chi_d(b)b\inv}\left(1+\frac{(\chi_d(b^{t-1})-\chi(b^t))P(w)}{b^t(1-b\inv)}\right)\mme^{\gamma}\tau+O((\tau\log\tau)^{\frac{1}{2}}).
\end{equation}

Finally, if $\alpha=\beta$, then $\frac{k}{\ell}=\frac{a}{b}$ and $t=1$. When $\chi_d(b)=1$, there is nothing to prove, when $\chi_d(b)=-1$, since 
\[m(\chi_d)=\frac{\mme^{-\gamma}\pi\mmi}{\mcg(\chi_d)}\sum_{n\leq \beta d}\chi_d(n)\geq \tau.\]
So combining Eq. \eqref{eq:mchidtau}, we get 
\[\frac{b+1}{b-1}\left(1+\frac{2P(u_0)}{b-1}\right)\geq 1+O(\tau^{-\fot}(\log\tau)^{\fot}).\]
Then Eq. \eqref{eq:pu0tau} follows immediately.

\section{\blue{Proof of Theorem \ref{thm1.3}}}\label{sec5}
As before, we only consider the fundamental discriminants $x^\delta<|d|\le x$ with $\delta=\frac{1}{100}$.
By Lemma \ref{polyatruc}, since $\chi_d$ is even, we have 
\begin{align*}m(\chi_d)&=\frac{1}{2\mme^\gamma}\max_{\alpha\in[0,1]}\bigg|\sum_{1\le|n|\le z}\frac{\chi_d(n)\mme(n\alpha)}{n}\bigg|+O(x^{-\delta})\\&\le\frac{1}{2\mme^\gamma}\max_{\alpha\in[0,1]}\bigg|\sum_{1\le|n|\le z\atop n\in S(y)}\frac{\chi_d(n)\mme(n\alpha)}{n}\bigg|+\frac{2}{\mme^\gamma}S_{y,x^{{21}/{40}}}(\chi_d)+O(x^{-\delta}).
\end{align*}

We set $y=\mme^{\sqrt3\tau+c}$ for some constant $c$. Define the set
$$\mathcal{C}_x^+(\tau):=\{d\in{\mathcal F}^+(x):\;S_{y,x^{{21}/{40}}}(\chi_d)\le1,\;m(\chi_d)>\tau\}.$$ 
By Lemma \ref{estimateforS} and Theorem B, the cardinality of $\mathcal{C}_x^+(\tau)$ satisfies the request of Theorem \ref{thm1.3}. Fix $\tau$. For any $d\in\mathcal{C}_x^+(\tau)$, choose $\alpha=N_{\chi_d}/|d|$. Let $a/b$ be the rational approximation of $\alpha$ such that $(a,b)=1$, $b\le\tau^{10}$ and $|\alpha-a/b|<1/(b\tau^{10})$. Using the definition of $\mathcal{C}_x^+(\tau)$ and then Lemma \ref{alphaapprox}, we have for $N:=1/|b\alpha-a|$,
\begin{align*}
    \tau<m(\chi_d)&\le\frac{1}{2\mme^\gamma}\bigg|\sum_{1\le|n|\le z\atop n\in {\mathcal S}(y)}\frac{\chi_d(n)\mme(n\alpha)}{n}\bigg|+O(1)\\
    &=\frac{1}{2\mme^\gamma}\bigg|\sum_{1\le|n|\le N\atop n\in {\mathcal S}(y)}\frac{\chi_d(n)\mme(na/b)}{n}\bigg|+O(\log\tau).   
\end{align*}
Let $\xi$ and $D$ be as in Lemma \ref{functupper} for $\chi_d$. Since  $m(\chi_d)>\tau$, Lemma  \ref{functupper} forces that $\chi_d$ is in the second case. That is, $D|d$ and $\xi$ is odd. So we have
\begin{align*}
m(\chi_d)&\le\frac{1}{2\mme^\gamma}\bigg|\sum_{1\le|n|\le N\atop n\in {\mathcal S}(y)}\frac{\chi_d(n)\mme(na/b)}{n}\bigg|+O(\log\tau)\\&\le\frac{1}{2\mme^\gamma}\log y\frac{2\mme^\gamma}{\sqrt D}(2/3))^{\omega(b/D)}+O(\tau^{0.86})\\
&=\frac{1}{\sqrt D}(\sqrt3\tau+c)(2/3))^{\omega(b/D)}+O(\tau^{0.86}).
\end{align*}
Again using $m(\chi_d)>\tau$, we have
$$\sqrt\frac{3}{ D}(2/3)^{\omega(b/D)}\ge1.$$ 
Note that $\xi\pmod D$ is odd implies $D\ge3$. Thus we have to choose $b=D=3$, which completes the first assertion of Theorem \ref{thm1.3}.

Now we have
$$\tau<m(\chi_d)\le\frac{1}{2\mme^\gamma}\bigg|\sum_{1\le|n|\le N\atop n\in {\mathcal S}(y)}\frac{\chi_d(n)\mme(na/3)}{n}\bigg|+O(\log\tau)=\frac{\sqrt3}{2\mme^\gamma}\bigg|\sum_{n\le N\atop n\in {\mathcal S}(y)}\frac{\chi_d(n)\big(\frac n3\big)}{n}\bigg|+O(\log\tau),$$
where we have used 
$$\mme(an/3)-\mme(-an/3)=i\sqrt3\bigg(\frac{an}{3}\bigg)=i\sqrt3\bigg(\frac{a}{3}\bigg)\bigg(\frac{n}{3}\bigg).$$
By the trivial upper bound and then Mertens' formula, we have
$$\bigg|\sum_{n\le N\atop n\in {\mathcal S}(y)}\frac{\chi_d(n)\big(\frac n3\big)}{n}\bigg|
\le\sum_{n\le N\atop n\in {\mathcal S}(y)}\frac{\big|\big(\frac n3\big)\big|}{n}=
\sum_{3\nmid n\le N\atop n\in {\mathcal S}(y)}\frac{1}{n}
\le\sum_{3\nmid n \in {\mathcal S}(y)}\frac{1}{n}=\frac23\mme^\gamma\log y+O(1)=\frac{2}{\sqrt3}\mme^\gamma\tau+O(1).$$
This combined with the last inequality gives
$$\tau<m(\chi_d)\le\frac{\sqrt3}{2\mme^\gamma}\bigg|\sum_{n\le N\atop n\in {\mathcal S}(y)}\frac{\chi_d(n)\big(\frac n3\big)}{n}\bigg|+O(\log\tau)\le\frac{\sqrt3}{2\mme^\gamma}\sum_{3\nmid n \in {\mathcal S}(y)}\frac{1}{n}\le\tau+O(\log\tau).$$
Then we have
\begin{align}&\sum_{3\nmid n\in {\mathcal S}(y)}\frac1n=\sum_{n\le N\atop3\nmid n\in {\mathcal S}(y)}\frac1n+O(\log\tau)\nonumber\\=&\bigg|\sum_{n\in {\mathcal S}(y)}\frac{\chi_d(n)\big(\frac n3\big)}{n}\bigg|+O(\log\tau)=\bigg|\sum_{n\le N\atop n\in {\mathcal S}(y)}\frac{\chi_d(n)\big(\frac n3\big)}{n}\bigg|+O(\log\tau).\label{logtauerr}\end{align}
Combining the first and the third one, we have
$$\prod_{3\neq p\le y}\bigg(1-\frac1p\bigg)^{-1}\bigg(1-\frac{\chi_d(p)\big(\frac p3\big)}{p}\bigg)=1+O\bigg(\frac{\log\tau}\tau\bigg).$$
Taking logarithm we have
$$\sum_{3\neq p\le y}\sum_{j\ge1}\frac{1-\chi_d(p^j)\big(\frac {p^j}3\big)}{jp^j}\ll\frac{\log\tau}\tau.$$
This is stronger than the second assertion in Theorem \ref{thm1.3} since 
$$1-\chi_d(p^j)\bigg(\frac {p^j}3\bigg)=\bigg|\chi_d(p^j)-\bigg(\frac {p^j}3\bigg)\bigg|.$$
In view of the truncating approximation
$$L\left(1,\chi_d(\frac{\cdot}{3})\right)=\sum_{n\le x^{21/40}}\frac{\chi_d(n)\big(\frac n3\big)}{n}+O(1),$$
to prove the last assertion in Theorem \ref{thm1.3}, it is sufficient to show 
\begin{align}\bigg|\sum_{n\le x^{21/40}\atop n\notin S(y)}\frac{\chi_d(n)\big(\frac n3\big)}{n}\bigg|+\bigg|\sum_{n> x^{21/40}\atop n\in S(y)}\frac{\chi_d(n)\big(\frac n3\big)}{n}\bigg|\ll\log\tau.\label{errorterm}\end{align}
It follows directly by applying the definition of ${\mathcal C_x^+(\tau)}$ to the first sum and Lemma \ref{friablesum} to the second. Thus we complete the proof of Theorem \ref{thm1.3}.

\section{\blue{Proof of Theorem \ref{thm1.4}}}\label{sec6}
As before, we may only consider the fundamental discriminants $x^{\frac{1}{100}}<d\le x$. Let $|\beta-k/\ell|\le 1/(\ell\tau^{10})$ with $l\le\tau^{10}$. By Lemma \ref{polyatruc}, the definition of ${\mathcal C}_x^+(\tau)$, and then Lemma \ref{alphaapprox}, we have
\begin{align*}\frac{2\pi {\rm i}}{{\mathcal G}(\chi_d)}\sum_{n\le\beta|d|}\chi_d(n)&=\sum_{1\le|n|\le x^{21/40}\atop n\in {\mathcal S}(y)}\frac{\chi_d(n)\mme(n\beta )}{n}+O(1)\\&=\sum_{1\le|n|\le \mme^{\sqrt3\tau u}\atop n\in {\mathcal S}(y)}\frac{\chi_d(n)\mme( nk/\ell)}{n}+O(\log\tau)\\&=2{\rm i}\sum_{n\le \mme^{\sqrt3\tau u}\atop n\in {\mathcal S}(y)}\frac{\chi_d(n)\sin(2\pi nk/\ell)}{n}+O(\log\tau).
\end{align*}
Since Eq. \eqref{logtauerr} implies
$$
\sum_{3\nmid n\in {\mathcal S}(y)}\frac{\big|\chi_d(n)-\big(\frac n3\big)\big|}{n}=\sum_{3\nmid n\in {\mathcal S}(y)}\frac{1-\chi_d(n)\big(\frac n3\big)}{n}\ll{\log\tau},$$
 we have
 \begin{align*}\frac{\pi }{{\mathcal G}(\chi_d)}\sum_{n\le\beta|d|}\chi_d(n)=&\sum_{n\le \mme^{\sqrt3\tau u}\atop n\in {\mathcal S}(y)}\frac{\chi_d(n)\sin(2\pi nk/\ell)}{n}+O(\log\tau)\\
=&\sum_{0\le j\le \sqrt3\tau u/\log3}\frac{\chi_d(3^j)}{3^j}\sum_{m\le \mme^{\sqrt3\tau u}/3^j\atop3\nmid m\in {\mathcal S}(y)}\frac{\chi_d(m)\sin(2\pi3^jmk/\ell)}{m}+O(\log\tau)\\
=&\sum_{0\le j\le \sqrt3\tau u/\log3}\frac{\chi_d(3^j)}{3^j}\sum_{m\le \mme^{\sqrt3\tau u}/3^j\atop3\nmid m\in {\mathcal S}(y)}\frac{\big(\frac m3\big)\sin(2\pi3^jmk/\ell)}{m}+O(\log\tau)\\
=&\sum_{0\le j\le \sqrt3\tau u/\log3}\frac{\chi_d(3^j)}{3^j}\sum_{m\le \mme^{\sqrt3\tau u}/3^j\atop m\in {\mathcal S}(y)}\frac{\big(\frac m3\big)\sin(2\pi3^jmk/\ell)}{m}+O(\log\tau).
\end{align*} 
Note that the last equation holds since $\big(\frac m3\big)=0$ whenever $3|m$.
Since $\sqrt3\big(\frac m 3\big)=2\sin(2\pi m/3)$ for any integer $m$, the inner sum equals
\begin{align*}
\frac{2}{\sqrt3}\sum_{m\le \mme^{\sqrt3\tau u}/3^j\atop m\in {\mathcal S}(y)}&\frac{
\sin(2\pi m/3)\sin(2\pi3^jmk/\ell)}{m}\\=&\frac{1}{\sqrt3}\sum_{m\le \mme^{\sqrt3\tau u}/3^j\atop  m\in {\mathcal S}(y)}\frac{\cos(2\pi m(3^jk/\ell-1/3))-\cos(2\pi m(3^jk/\ell+1/3))}{m}.
\end{align*}
	If $\ell=3^v$ for some $v\ge1$, $(k-\big(\frac k3\big))/3$ is always  an integer while $(k+\big(\frac k3\big))/3$ is not, then the term $j=v-1$ contributes to the main term 
	$$\frac{1}{\sqrt3}\Big(\frac k3\Big)\sum_{m\le \mme^{\sqrt3\tau u}/3^j\atop  m\in {\mathcal S}(y)}\frac1m=\frac{1}{\sqrt3}\Big(\frac k3\Big)\mme^\gamma P(u)\log y+O(\log\tau)$$
	by Lemma \ref{friablenumber},
	while the other terms only lead to error terms. So in this case we have 
	$$\frac{\pi }{{\mathcal G}(\chi_d)}\sum_{n\le\beta|d|}\chi_d(n)=\mme^\gamma\frac{\chi_d(3^{v-1})}{3^{v-1}}\Big(\frac k3\Big)\tau P(u)+O(\log\tau).$$
	If $\ell$ is not a power of $3$, then $3^jk/\ell\pm1/3$ is never an integer. Thus by Lemma \ref{remove exp cor}, the inner sum only contributes to the error term
	$$\sum_{m\le \mme^{\sqrt3\tau u}/3^j\atop  m\in {\mathcal S}(y)}\frac{\cos(2\pi m(3^jk/\ell-1/3))-\cos(2\pi m(3^jk/\ell+1/3))}{m}\ll\log\tau,$$
	and then we have
		$$\frac{\pi }{{\mathcal G}(\chi_d)}\sum_{n\le\beta|d|}\chi_d(n)\ll\log\tau.$$
		We complete the proof of Theorem \ref{thm1.4}.
\section*{Acknowledgements}
The authors would like to thank Professor Youness Lamzouri for his valuable comments on the previous version of this article. Hao Zhang was supported by Fundamental Research Funds for the Central Universities (Grant No. 531118010622).



\end{document}